\documentclass[12pt]{amsart}
\usepackage{amsmath,amsfonts,amstext,amsthm,epsfig}

\newtheorem{Theorem}{Theorem}[section]
\newtheorem{Lemma}[Theorem]{Lemma}%[section]
\newtheorem{Prop}[Theorem]{Proposition}%[section]
\newtheorem{Rem}[Theorem]{Remark}%[section]
\newtheorem{Hyp}[Theorem]{Hypothesis}%[section]
\makeatletter
    \addtocounter{section}{0}
    
     \@addtoreset{equation}{section}
\makeatother
\def\N{\mathbb N}
\def\R{\mathbb R}

\begin{document}

\title[Dimension of invariant sets]
{Dimension of attractors and invariant sets in reaction diffusion equations}%
\author{Martino Prizzi}

\address{Martino Prizzi, Universit\`a di Trieste, Dipartimento di
Matematica e Informatica, Via Valerio 12/1, 34127 Trieste, Italy}%
\email{mprizzi@units.it}%
%\thanks{}

\date{\today}%
%\dedicatory{}%
%\commby{}%
% ----------------------------------------------------------------
\begin{abstract} Under fairly general assumptions, we prove that every compact invariant set $\mathcal I$ of the semiflow generated
by the semilinear reaction diffusion equation
\begin{equation*}
\begin{aligned}
u_t+\beta(x)u-\Delta u&=f(x,u),&&(t,x)\in[0,+\infty[\times\Omega,\\
u&=0,&&(t,x)\in[0,+\infty[\times\partial\Omega
\end{aligned}\end{equation*}
in $H^1_0(\Omega)$ has finite Hausdorff dimension.
Here $\Omega$ is an arbitrary, possibly unbounded, domain in $\R^3$ and $f(x,u)$ is a nonlinearity of subcritical growth.
The nonlinearity $f(x,u)$ needs not to satisfy any dissipativeness assumption and the
invariant subset $\mathcal I$ needs not to be an an attractor. If $\Omega$ is regular, $f(x,u)$
is dissipative and $\mathcal I$ is the global attractor, we give an explicit bound on the Hausdorff dimension of
$\mathcal I$ in terms of the structure parameter of the equation.
\end{abstract}
\maketitle
% ----------------------------------------------------------------

\section{Introduction}
In this paper we consider the reaction diffusion equation
\begin{equation}\label{equation1}
\begin{aligned}
u_t+\beta(x)u-\Delta
u&=f(x,u),&&(t,x)\in[0,+\infty[\times\Omega,\\
u&=0,&&(t,x)\in[0,+\infty[\times\partial\Omega
\end{aligned}\end{equation}
Here $\Omega$ is an arbitrary (possibly unbounded) open set in
$\R^3$, $\beta(x)$ is a potential such that the operator
$-\Delta+\beta(x)$ is positive, and $f(x,u)$ is a nonlinearity of
subcritical growth (i.e. of polynomial growth strictly less than
five). The assumptions on $\beta(x)$ and $f(x,u)$ will be made more
precise in Section 2 below. Under such assumptions, equation
(\ref{equation1}) generates a local semiflow $\pi$ in the space
$H^1_0(\Omega)$. Suppose that the semiflow $\pi$ admits a compact
invariant set $\mathcal I$ (i.e. $\pi(t,\mathcal I)=\mathcal I$ for all $t\geq0$). 
We do not make any structure assumption on the nonlinearity $f(x,u)$ and therefore
 we do not assume that $\mathcal I$ is
the global attractor of equation (\ref{equation1}): for example,
$\mathcal I$ can be an unstable invariant set detected by Conley
index arguments (see e.g. \cite{Pr}). Our aim is to prove that
$\mathcal I$ has finite Hausdorff dimension and to give an
explicit estimate of its dimension. When $\Omega$ is a bounded
domain and $f(x,u)$ satisfies suitable dissipativeness conditions, 
the existence of a finite dimensional compact global
attractor for (\ref{equation1})  is  a classical achievement (see e.g. \cite{BV2, Lady, Tem}).
When $\Omega$ is unbounded, new difficulties arise due to the lack
of compactness of the Sobolev embeddings. These difficulties can
be overcome in several ways: by introducing weighted spaces (see
e.g. \cite{BV,EZ}), by developing suitable tail-estimates (see
e.g. \cite{W, PR1}), by exploiting comparison arguments (see e.g.
\cite{ACDR}). Concerning the finite dimensionality of the
attractor, in \cite{BV,EZ,W} and other similar works the potential
$\beta(x)$ is always assumed to be just a positive constant.
In \cite{AMR} Arrieta et al.  considered for the first time the
case of a sign-changing potential. In their results the invariant
set $\mathcal I$ does not need to be an attractor; however they
need to make some structure assumptions on $f(x,u)$ which
essentially resemble the conditions ensuring the existence of the
global attractor. Moreover,  in \cite{AMR} the invariant set is
a-priori assumed to be bounded in the $L^\infty$-norm. In concrete
situations, such a-priori estimate can be obtained through
elliptic regularity combined with some comparison argument. This
in turn requires to make some regularity assumption on the boundary of
$\Omega$. In this paper we do not make any structure assumption on
the nonlinearity $f(x,u)$, neither do we assume $\partial\Omega$
to be regular. Our only assumption is that the mapping $h\mapsto
(\partial_u f(x,0))_+h$ has to be a relatively form compact perturbation of
$-\Delta+\beta(x)$. This can be achieved, e.g., by assuming that
$\partial_u f(x,0)$ can be estimated from above by some positive
$L^r$ function, $r>3/2$. Under this assumption, we shall prove
that $\mathcal I$ has finite Hausdorff dimension. Also, we give an
explicit estimate of the dimension of $\mathcal I$, involving the
number $\mathcal N$ of negative eigenvalues of the operator
$-\Delta+\beta(x)-\partial_u f(x,0)$. When $\Omega$ has a regular
boundary, we can explicitly estimate $\mathcal N$ by mean of
Cwickel-Lieb-Rozenblum inequality (see \cite{RoSo}); as a
consequence, if we also assume that  $f(x,u)$ is dissipative, we recover the result of
Arrieta et al. \cite{AMR}.

The paper is organized as follows. In Section 2 we introduce
notations, we state the main assumptions and we collect some
preliminaries about the semiflow generated by equation
(\ref{equation1}). In Section 3 we prove that the semiflow
generated by equation (\ref{equation1}) is uniformly
$L^2$-differentiable on any compact invariant set $\mathcal I$. In
Section 4 we recall the definition of Hausdorff dimension and we
prove that any compact invariant set $\mathcal I$ has finite
Hausdorff dimension in $L^2(\Omega)$ as well as in
$H^1_0(\Omega)$. In Section 5 we compute the number of negative
eigenvalues of the operator $-\Delta+\beta(x)-\partial_u f(x,0)$
by mean of Cwickel-Lieb-Rozenblum inequality. In Section 6 we
specialize our result to the case of a dissipative equation and we
recover the result of Arrieta et al. \cite{AMR}.

The results contained in this paper continue to hold if one
replaces $-\Delta$ with the general second order
elliptic operator in divergence form 
$-\sum_{i,j=1}^3\partial_{x_i}(a_{ij}(x)\partial_{x_j})$.

%-------------------------------------------------------------------

\section {Notation, preliminaries and remarks}

Let $\sigma\geq 1$. We denote by $L^\sigma_{\rm u}(\R^N)$ the set of measurable functions $\omega\colon \R^N\to\R$ such that
\begin{equation*}
|\omega|_{L^\sigma_{\rm u}}:=\sup_{y\in\R^N}\left(\int_{B(y)}|\omega(x)|^\sigma\,dx\right)^{1/\sigma}<\infty,
\end{equation*}
where, for $y\in\R^N$, $B(y)$ is the open unit cube in $\R^N$ centered at $y$.

In this paper we assume throughout that $N=3$, and we fix an open (possibly unbounded) set $\Omega\subset\R^3$.
We denote by $M_B$ the constant of the Sobolev embedding $H^1(B)\subset  L^6(B)$, where $B$ is any open unit
cube in $\R^3$. Moreover, for $2\leq q\leq6$, we denote by $M_q$ the constant of the Sobolev embedding $H^1(\R^3)\subset L^q(\R^3)$.

\begin{Prop}\label{prop1}
Let $\sigma>3/2$ and let $\omega\in L^\sigma_{\rm u}(\R^3)$. Set $\rho:= 3/2\sigma$. Then, for every $\epsilon>0$ and for every $u\in H^1_0(\Omega)$,
\begin{equation}
\int_{\Omega}|\omega(x)| |u(x)|^2\,dx\leq  |\omega|_{L^\sigma_{\rm u}}
\left(\rho\epsilon M_B^{2}|u|_{H^1}^{2}+(1-\rho)\epsilon^{-\rho/(1-\rho)}|u|_{L^2}^{2}\right).
\end{equation}
Moreover, for every $u\in H^1_0(\Omega)$,
\begin{equation}
\int_{\Omega}|\omega(x)| |u(x)|^2\,dx\leq M_B^{2\rho} |\omega|_{L^\sigma_{\rm u}}|u|_{H^1}^{2\rho}|u|_{L^2}^{2(1-\rho)}.
\end{equation}
\end{Prop}
\begin{proof}See the proof of Lemma 3.3 in \cite{PR2}.\end{proof}

Let $\beta\in L^\sigma_{\rm u}(\R^3)$, with $\sigma>3/2$. Let us consider the following bilinear form defined on the space $H^1_0(\Omega)$:
\begin{equation}
a(u,v):=\int_\Omega\nabla u(x)\cdot\nabla v(x)\,dx+\int_\Omega\beta(x)u(x)v(x)\,dx,\quad u,v\in H^1_0(\Omega)
\end{equation}
Our first assumption is the following:
\begin{Hyp}\label{hyp1}
There exists $\lambda_1>0$ such that
\begin{equation}
\int_\Omega |\nabla u(x)|^2\,dx+\int_\Omega\beta(x)|u(x)|^2\,dx\geq \lambda_1 |u|^2_{L^2},\quad u\in H^1_0(\Omega).\label{f10}
\end{equation}
\end{Hyp}

\begin{Rem} Conditions on $\beta(x)$ under which Hypothesis \ref{hyp1} is satisfied are expounded e.g. in \cite{AB1,AB2}. 
\end{Rem}

As a consequence of (\ref{f10}) and Proposition \ref{prop1}, we have:
\begin{Prop}\label{prop2} There exist two positive constants $\lambda_0$ and $\Lambda_0$ such that
\begin{equation}
\lambda_0|u|^2_{H^1}\leq\int_\Omega |\nabla u(x)|^2\,dx+\int_\Omega\beta(x)|u(x)|^2\,dx\leq \Lambda_0 |u|^2_{H^1},\quad u\in H^1_0(\Omega).
\end{equation}
The constants $\lambda_0$ and $\Lambda_0$ can be computed explicitly in terms of $\lambda_1$, $M_B$ and $|\beta|_{L^\sigma_{\rm u}}$.
\end{Prop}
\begin{proof}Cf Lemma 4.2 in \cite{PR1}\end{proof}
It follows from Proposition \ref{prop2} that the bilinear form $a(\cdot,\cdot)$ defines a scalar product in $H^1_0(\Omega)$,
equivalent to the standard one.
According to the results of Section 4 in \cite{PR1},  $a(\cdot,\cdot)$ induces a positive selfadjoint operator $A$ in the space $L^2(\Omega)$.
$A$ is uniquely determined by the relation
\begin{equation}
\langle A u,v\rangle_{L^2}=a(u,v),\quad u\in D(A),v\in
H^1_0(\Omega).
\end{equation}
Notice that $A u=-\Delta u+\beta u$ in the sense of distributions,
and $u\in D(A)$ if and only if $-\Delta u+\beta u\in L^2(\Omega)$.
Set $X:=L^2(\Omega)$, and let $(X^{\alpha})_{\alpha\in\R}$ be the
scale of fractional power spaces associated with $A$ (see Section
2 in \cite{PR1} for a short, self-contained, description of this
scale of spaces). Here we just recall that $X^0=L^2(\Omega)$,
$X^1=D(A)$, $X^{1/2}=H^1_0(\Omega)$ and $X^{-\alpha}$ is the dual
of $X^\alpha$ for $\alpha\in]0,+\infty[$. For
$\alpha\in]0,+\infty[$, the space $X^\alpha$ is a Hilbert space
with respect to the scalar product
\begin{equation*}
\langle u,v\rangle_{X^\alpha}:=\langle A^\alpha
u,A^\alpha v\rangle_{L^2},\quad u,v\in X^\alpha.
\end{equation*}
Also, the space $X^{-\alpha}$ is a Hilbert space with respect to
the scalar product $\langle \cdot,\cdot\rangle_{X^{-\alpha}}$ dual
to the scalar product  $\langle \cdot,\cdot\rangle_{X^{\alpha}}$,
i.e.
\begin{equation*}
\langle u',v'\rangle_{X^{-\alpha}}=\langle R^{-1}_\alpha
u',R^{-1}_\alpha v'\rangle_{X^\alpha},\quad u,v\in X^{-\alpha},
\end{equation*}
where $R_\alpha\colon X^\alpha\to X^{-\alpha}$ is the Riesz
isomorphism $u\mapsto\langle\cdot,u\rangle_{X^\alpha}$. Finally,
for every $\alpha\in\R$, $A$ induces a selfadjoint  operator
$A_{(\alpha)}\colon X^{\alpha+1}\to X^\alpha$, such that
$A_{(\alpha')}$ is an extension of $A_{(\alpha)}$ whenever
$\alpha'\leq\alpha$, and $D(A_{(\alpha)}^\beta)=X^{\alpha+\beta}$
for $\beta\in[0,1]$. If $\alpha\in[0,1/2]$, $u\in X^{1-\alpha}$
and $v\in X^{1/2}\subset X^\alpha$, then
\begin{equation}
\langle v,
A_{(-\alpha)}u\rangle_{(X^{\alpha},X^{-\alpha})}=\langle u,
v\rangle_{X^{1/2}}=a(u,v).
\end{equation}

\begin{Lemma} Let $(X^\alpha)_{\alpha\in\R}$ be as above.
\begin{enumerate}
\item If $p\in[2,6[$, then $X^\alpha\subset L^p(\Omega)$ for $\alpha\in ] 3(p-2)/4p,1/2]$. Accordingly, if $q\in]6/5,2]$, then
$L^q(\Omega)\subset X^{-\alpha}$ for $\alpha\in\, ]
3(2-q)/4q,1/2]$.
\item If $\sigma>3/2$ and $\omega\in L^\sigma_{\rm u}(\Omega)$, then
the assignment $u\mapsto \omega u$ defines a bounded linear map
from $X^{1/2}$ to $X^{-\alpha}$ for $\alpha\in\, ]3/4\sigma,1/2]$.
\end{enumerate}
\end{Lemma}
\begin{proof} See Lemmas 5.1 and 5.2 and the proof of Proposition
5.3 in \cite{PR1}.
\end{proof}

Our second assumption is the following:

\begin{Hyp}\label{hyp2}\
\begin{enumerate}
\item $f\colon\Omega\times\R\to\R$ is such that, for every $u\in\R$, $f(\cdot,u)$ is measurable and, for a.e. $x\in\Omega$,
$f(x,\cdot)$ is of class $C^2$;
\item $f(\cdot,0)\in L^q(\Omega)$, with $6/5<q\leq 2$ and $\partial_uf(\cdot,0)\in L^\sigma_{\rm u}(\R^3)$, with $\sigma>3/2$;
\item  there exist  constants $C$ and $\gamma$, with $C>0$ and $2\leq\gamma<3$ such that
$|\partial_{uu} f(x,u)|\leq C(1+|u|^\gamma)$. Notice that, in view of Young's inequality, the requirement $\gamma\geq2$ is not restrictive.
\end{enumerate}\end{Hyp}

We introduce the Nemitski operator $\hat f$ which associates with
every function $u\colon\Omega\to\R$ the function $\hat
f(u)(x):=f(x,u(x))$.
\begin{Prop} Assume $f$ satisfies Hypothesis \ref{hyp2}. Let $\alpha$ be such that
\begin{equation}
\frac12>\alpha>\max\left\{\frac{\gamma-1}4,\,\frac34\frac{2-q}q,\,\frac3{4\sigma}\right\}.
\end{equation}
Then the assignment $u\mapsto\mathbf f(u)$, where
\begin{equation}
\langle v,\mathbf f(u)
\rangle_{(X^\alpha,X^{-\alpha})}:=\int_\Omega\hat f(u)(x)v(x)\,dx,
\end{equation}
defines a map  $\mathbf f\colon X^{1/2}\to X^{-\alpha}$ which is
Lipschitzian on bounded sets.
\end{Prop}
\begin{proof}
See the proof of Proposition 5.3 in \cite{PR1}.
\end{proof}
Setting $\mathbf X:=X^{-\alpha}$ and $\mathbf A:=A_{(-\alpha)}$,
we have that $\mathbf X^{\alpha+1/2}=X^{1/2}$. We can rewrite
equation (\ref{equation1}) as an abstract parabolic problem in the
space $\mathbf X$, namely
\begin{equation}\label{equation2}
\dot u+\mathbf A u=\mathbf f(u).
\end{equation}

By results in \cite{He}, equation (\ref{equation2}) has a unique
\emph{mild solution} for every initial datum $u_0\in \mathbf
X^{\alpha+1/2}=H^1_0(\Omega)$, satisfying the \emph{variation of
constants} formula
\begin{equation}
u(t)=e^{-\mathbf A t}u_0+\int_0^te^{-\mathbf A(t-s)}\mathbf
f(u(s))\,ds,\quad t\geq 0.
\end{equation}

It follows that (\ref{equation2}) generates a local semiflow $\pi$
in the space $H^1_0(\Omega)$. Moreover, if $u(\cdot)\colon[0,T[\to
\mathbf X^{\alpha+1/2}$ is a mild solution of (\ref{equation2}),
then $u(t)$ is differentiable into $\mathbf X^{\alpha+1/2}=H^1_0(\Omega)$ for $t\in\,]0,T[$, and it satisfies equation
(\ref{equation2}) in $\mathbf X=X^{-\alpha}\subset
{H^{-1}}(\Omega)$. In particular, $u(\cdot)$ is a \emph{weak
solution} of (\ref{equation1}).

Assume now that $\mathcal I\subset H^1_0(\Omega)$ is a compact
invariant set for the semiflow $\pi$ generated by
(\ref{equation2}). If $\mathcal B$ is a Banach space such that
$H^1_0(\Omega)\subset \mathcal B$, we define
\begin{equation}
|\mathcal I|_{\mathcal B}:=\max\left\{|u|_{\mathcal B}\mid
u\in\mathcal I \right\}.
\end{equation}

We end this section with a technical lemma that will be used later.

\begin{Lemma}\label{liplip} For every $T>0$ there exists a constant $L(T)$ such that,
whenever $u_0$ and $v_0\in \mathcal I$, setting $u(t):=\pi( t, u_0)$ and $v(t):=\pi( t,v_0)$, $t\geq 0$, the following estimate holds:
\begin{equation}
|u(t)-v(t)|_{H^1}\leq L(T)t^{-(\alpha +1/2)}|u_0-v_0|_{L^2}, \quad t\in]0,T].
\end{equation}
The constant $L(T)$  depends only on  $|\mathcal I |_{H^1}$,   and on the constants of Hypotheses \ref{hyp1} and \ref{hyp2}.
\end{Lemma}
\begin{proof}
We have
\begin{equation*}
u(t)-v(t)=e^{-\mathbf A t}(u_0-v_0)+\int_0^t e^{-\mathbf A (t-s)}(\mathbf f (u(s))-\mathbf f (v(s)))\,ds;
\end{equation*}
it follows that
\begin{multline*}
|u(t)-v(t)|_{\mathbf X^{\alpha+1/2}}\leq t^{-(\alpha +1/2)}|u_0-v_0|_{\mathbf X}+
\int_0^t (t-s)^{-(\alpha +1/2)}|\mathbf f (u(s))-\mathbf f (v(s))|_{\mathbf X}\,ds\\
\leq  t^{-(\alpha +1/2)}|u_0-v_0|_{\mathbf X}+\int_0^t (t-s)^{-(\alpha +1/2)}C(|\mathcal I|_{H^1})|u(s)-v(s)|_{\mathbf X^{\alpha+1/2}}\,ds.
\end{multline*}
By Henry's inequality \cite[Theorem 7.1.1]{He}, this implies that
\begin{equation*}
|u(t)-v(t)|_{\mathbf X^{\alpha+1/2}}\leq L(T)t^{-(\alpha +1/2)}|u_0-v_0|_{\mathbf X}, \quad t\in]0,T],
\end{equation*}
and the thesis follows.
\end{proof}

%-----------------------------------------------------------------

\section {Uniform differentiability}
In this section we prove some technical results which will allow
us to apply the methods of \cite{Tem} for proving finite
dimensionality of compact invariant sets. We assume throughout
that $\mathcal I\subset H^1_0(\Omega)$ is a compact invariant set
of the semiflow $\pi$ generated by equation (\ref{equation2}).

\begin{Lemma}\label{cont-dip} There exists a constant $K$ such that, whenever $u_0$ and
$v_0\in \mathcal I$, setting $u(t):=\pi( t, u_0)$ and $v(t):=\pi( t,v_0)$, $t\geq 0$, the following estimate holds:
\begin{equation}
|u(t)-v(t)|_{L^2}^2+\lambda_0\int_0^t|u(s)-v(s)|^2_{H^1}\,ds\leq e^{Kt}|u_0-v_0|_{L^2}^2
\end{equation}
The constant $K$  depends only on  $|\mathcal I |_{H^1}$,  on $\lambda_0$ and $\Lambda_0$ (see Proposition \ref{prop2}),
on $|\partial_u f(\cdot,0)|_{L^\sigma_{\rm u}}$, and on the constants $C$ and $\gamma$ (see Hypothesis \ref{hyp2}).
\end{Lemma}
\begin{proof}
Set $z(t)=u(t)-v(t)$. Then
\begin{multline*}
\frac12\frac d{dt}|z(t)|_{L^2}^2+\int_\Omega|\nabla z(t)(x)|^2\,dx+\int_\Omega\beta(x)|z(t)(x)|^2\,dx\\
=\int_\Omega\left(f(x,u(t)(x))-f(x,v(t)(x)\right) z(t)(x)\,dx.
\end{multline*}
It follows from Proposition \ref{prop2} and Hypothesis \ref{hyp2} that
\begin{multline*}
\frac12\frac d{dt}|z(t)|_{L^2}^2+\lambda_0|z(t)|_{H^1}^2\leq\int_\Omega |\partial_u f(x,0)||z(t)(x)|^2\,dx\\
+C'\int_\Omega\left(1+|u(t)(x)|^{\gamma+1}+|v(t)(x)|^{\gamma+1}\right)|z(t)(x)|^2\,dx\\
\leq \int_\Omega |\partial_u f(x,0)||z(t)(x)|^2\,dx+C'|z(t)|_{L^2}^2\\
+C'\left(|u(t)|_{L^6}^{\gamma+1}+|v(t)|_{L^6}^{\gamma+1}\right)|z(t)|_{L^{12/(5-\gamma)}}^2,
\end{multline*}
where $C'$ is a constant depending only on $C$ and $\gamma$. Notice that $2<12/(5-\gamma)<6$. Therefore, by interpolation,
we get that for every $\epsilon>0$ there exists a constant $c_\epsilon>0$ such that
\begin{equation}\label{interpol1}
|z(t)|_{L^{12/(5-\gamma)}}^2\leq
\epsilon|z(t)|_{H^1}^2+c_\epsilon|z(t)|_{L^2}^2.
\end{equation}
Now (\ref{interpol1}) and Proposition \ref{prop1} imply that, for every $\epsilon>0$,  there exists a
constant $C'_\epsilon$, depending on $C'$, $|\mathcal I |_{H^1}$ and $\epsilon$, such that
\begin{equation}\label{ineq1}
\frac12\frac d{dt}|z(t)|_{L^2}^2+\lambda_0|z(t)|_{H^1}^2\leq \epsilon|z(t)|_{H^1}^2+C'_\epsilon|z(t)|_{L^2}^2.
\end{equation}
Now choosing $\epsilon=\lambda_0/2$ and multiplying (\ref{ineq1})
by $e^{-2C'_\epsilon t}$ we get
\begin{equation}\label{ineq2}
\frac d{dt}(e^{-2C'_\epsilon t}|z(t)|_{L^2}^2)+\lambda_0e^{-2C'_\epsilon t}|z(t)|_{H^1}^2\leq 0.
\end{equation}
Integrating (\ref{ineq2}) we obtain the thesis.
\end{proof}

Let $\bar u(\cdot)\colon\R\to H^1_0(\Omega)$ be a full bounded solution of (\ref{equation2}) such that $\bar u(t)\in\mathcal I$ for $t\in\R$.
Let us consider the non autonomous linear equation
\begin{equation}\label{equation3}
\begin{aligned}
u_t+\beta(x)u-\Delta
u&=\partial_uf(x,\bar u(t))u,&&(t,x)\in[0,+\infty[\times\Omega,\\
u&=0,&&(t,x)\in[0,+\infty[\times\partial\Omega
\end{aligned}\end{equation}
We introduce the following bilinear form defined on on the space $H^1_0(\Omega)$:
\begin{multline}\label{non-aut-form}
a(t;u,v):=\int_\Omega\nabla u(x)\cdot\nabla
v(x)\,dx\\+\int_\Omega\beta(x)u(x)v(x)\,dx- \int_\Omega
\partial_uf(x,\bar u(t)(x))u(x)v(x)\,dx,\quad u,v\in
H^1_0(\Omega).
\end{multline}

\begin{Prop}\label{hyp-Fuje-Tanabe}
There exist constants $\kappa_i>0$, $i=1$, \dots, $4$, such that:
\begin{enumerate}
\item $|a(t;u,v)|\leq \kappa_1|u|_{H^1}|v|_{H^1}$, $u,v\in H^1_0(\Omega)$, $t\in\R$;
\item $|a(t;u,u)|\geq \kappa_2|u|_{H^1}^2-\kappa_3|u|_{L^2}^2$, $u\in H^1_0(\Omega)$, $t\in\R$;
\item $|a(t;u,v)-a(s;u,v)|\leq\kappa_4|t-s||u|_{H^1}|v|_{H^1}$, $u,v\in H^1_0(\Omega)$, $t,s\in\R$.
\end{enumerate}
\end{Prop}
\begin{proof}
Properties (1) and (2) follow from Hypothesis \ref{hyp2} and
Proposition \ref{prop1}. In order to prove point (3), we first
observe that, by Theorem 3.5.2 in \cite{He} (and its proof), the
function ${\bar u}(\cdot)$ is differentiable into $H^1_0(\Omega)$,
with $|\dot {\bar u}(\cdot)|_{H^1}\leq L$, where $L$ is a constant
depending on $|\mathcal I|_{H^1}$ and on the constants in
Hypotheses \ref{hyp1} and \ref{hyp2}. Therefore we have:
\begin{multline}
|a(t;u,v)-a(s;u,v)|\leq\int_\Omega|\partial_u f(x,\bar
u(t)-\partial_u f(x,\bar u(s))||u(x)||v(x)|\,dx\\ \leq \int_\Omega
C(1+|\bar u(t)(x)|^\gamma+|\bar u(s)(x)|^\gamma)|\bar u(t)(x)-\bar
u(s)(x) ||u(x)||v(x)| \,dx\\ \leq C'(1+|\bar
u(t)|_{H^1}^\gamma+|\bar u(s)|_{H^1}^\gamma)|\bar u(t)-\bar
u(s)|_{H^1}|u|_{H^1}|v|_{H^1}\\ \leq C'(1+2|\mathcal I
|_{H^1}^\gamma)L|t-s||u|_{H^1}|v|_{H^1},
\end{multline}
and the proof is complete.\end{proof}

Now let $A(t)$ be the self-adjoint operator determined by the
relation
\begin{equation}\label{non-aut-op}
\langle A(t)u,v\rangle_{L^2}=a(t;u,v),\quad u\in D(A(t)),v\in
H^1_0(\Omega).
\end{equation}
We can apply Theorem 3.1 in \cite{FT} and get:

\begin{Prop}\label{sol-oper}
There exists a two parameter family of bounded linear operators
$U(t,s)\colon L^2(\Omega)\to L^2(\Omega)$, $t\geq s$, such that:
\begin{enumerate}
\item $U(s,s)=I$ for all $s\in\R$, and $U(t,s)U(s,r)=U(t,r)$ for all $t\geq s\geq
r$;
\item $U(t,s)h_0\in D(A(t))$ for all $h_0\in L^2(\Omega)$ and
$t>s$;
\item for every $h_0\in L^2(\Omega)$ and $s\in\R$, the map $t\mapsto
U(t,s)h_0$ is differentiable into $L^2(\Omega)$ for $t>s$, and
\begin{equation*}
\frac\partial{\partial t}U(t,s)h_0=-A(t)U(t,s)h_0.
\end{equation*}
\end{enumerate}
In particular, $U(t,s)h_0$ is a \emph{weak solution} of
(\ref{equation3}). \qed
\end{Prop}

Given $\bar u_0\in\mathcal I$, we take a full bounded solution $\bar u(\cdot)$ of (\ref{equation2}), whose trajectory is contained in $\mathcal I$,
and such that $\bar u(0)=\bar u_0$.
Then we define
\begin{equation}\label{def-diff}
\mathcal U(\bar u_0;t):=U(t,0),\quad t\geq 0,
\end{equation}
where $U(t,s)$ is the family of operators given by Proposition
\ref{sol-oper}. Notice that $\mathcal U(\bar u_0;t)$ does not
depend on the choice of $\bar u(\cdot)$, due to forward uniqueness
for equation (\ref{equation2}).
\begin{Prop}\label{unif-diff-1}
For every $t\geq0$,
\begin{equation}
\sup_{\bar u_0\in\mathcal I}\|\mathcal U(\bar u_0;t)\|_{\mathcal L(L^2,L^2)}<+\infty.
\end{equation}
\end{Prop}
\begin{proof}
Let $\bar u_0\in\mathcal I$ and $h_0\in L^2(\Omega)$. Set $h(t):=\mathcal U(\bar u_0;t)h_0$. Then, by property (3)
of Proposition \ref{sol-oper}, for $t>0$ we have
\begin{multline*}
\frac d{dt}\frac12|h(t)|_{L^2}^2+\int_\Omega|\nabla h(t)(x)|^2\,dx+\int_\Omega\beta(x)|h(t)(x)|^2\,dx\\
=\int_\Omega \partial_uf(x,\bar u(t)(x))|h(t)(x)|^2\,dx,
\end{multline*}
where $\bar u(\cdot)$ is a full bounded solution of (\ref{equation2}), whose trajectory is contained in $\mathcal I$, and such that $\bar u(0)=\bar u_0$.
It follows from Hypothesis \ref{hyp2} and Propositions \ref{prop1} and \ref{prop2} that for all $\epsilon>0$
\begin{multline*}
\frac d{dt}\frac12|h(t)|_{L^2}^2+\lambda_0 |h(t)|_{H^1}^2\\ \leq
\int_\Omega \partial_uf(x,0)|h(t)(x)|^2\,dx +\int_\Omega(
\partial_uf(x,\bar u(t)(x))-\partial_uf(x,0))|h(t)(x)|^2\,dx\\
\leq \epsilon |h(t)|_{H^1}^2+c_\epsilon |h(t)|_{L^2}^2+
\int_\Omega C(1+|\bar u(t)(x)|^\gamma)|\bar
u(t)(x)||h(t)(x)|^2\,dx\\ \leq \epsilon |h(t)|_{H^1}^2+c_\epsilon
|h(t)|_{L^2}^2+ \int_\Omega C'(1+|\bar
u(t)(x)|^{\gamma+1})|h(t)(x)|^2\,dx\\ \leq \epsilon
|h(t)|_{H^1}^2+(c_\epsilon+C') |h(t)|_{L^2}^2+C'|\bar
u(t)|_{L^6}^{\gamma+1}|h(t)|_{L^{12/(5-\gamma)}}^2.
\end{multline*}
Since $2<12/(5-\gamma)<6$, by interpolation
we get that for every $\epsilon>0$ there exists a constant $c'_\epsilon>0$ such that
\begin{equation*}
|h(t)|_{L^{12/(5-\gamma)}}^2\leq \epsilon|h(t)|_{H^1}^2+c'_\epsilon|h(t)|_{L^2}^2.
\end{equation*}
Therefore we have
\begin{equation}\label{gulp}
\frac d{dt}\frac12|h(t)|_{L^2}^2+\lambda_0 |h(t)|_{H^1}^2\leq \epsilon |h(t)|_{H^1}^2+C''(\epsilon,|\mathcal I|_{H^1})|h(t)|_{L^2}^2.
\end{equation}
Choosing $\epsilon=\lambda_0$ and integrating (\ref{gulp}) we
obtain
\begin{equation*}
|h(t)|_{L^2}^2\leq e^{2C''(\lambda_0,|\mathcal I|_{H^1})t}|h_0|_{L^2}^2
\end{equation*}
and the thesis follows.
\end{proof}

\begin{Prop}\label{unif-diff-2}
For every $t\geq 0$,
\begin{equation}
\lim_{\epsilon\to 0}\sup_{\substack{\bar u_0,\bar v_0\in\mathcal I\\0<|\bar u_0-\bar v_0|_{L^2}<\epsilon}}
\frac{|\pi(t,\bar v_0)-\pi(t,\bar u_0)-\mathcal U(\bar u_0;t)(\bar v_0-\bar u_0)|_{L^2}}{|\bar v_0-\bar u_0|_{L^2}}=0.
\end{equation}
\end{Prop}
\begin{proof}
Let $\bar u_0,\bar v_0\in\mathcal I$. Set $\bar u(t):=\pi( t, \bar u_0)$,  $\bar v(t):=\pi( t,\bar v_0)$ and
$\theta(t):=\bar v(t)-\bar u(t)-\mathcal U(\bar u_0;t)(\bar v_0-\bar u_0)$, $t\geq 0$.
A computation using property (3) of Proposition \ref{sol-oper} shows that, for $t>0$,
\begin{multline*}
\frac d{dt}\frac12|\theta(t)|_{L^2}^2+\int_\Omega|\nabla \theta(t)(x)|^2\,dx+\int_\Omega\beta(x)|\theta(t)(x)|^2\,dx\\
=\int_\Omega \partial_uf(x,\bar u(t)(x))|\theta(t)(x)|^2\,dx\\
+\int_\Omega \left(f(x,\bar v(t)(x))-f(x,\bar u(t)(x))-\partial_uf(x,\bar u(t)(x))(\bar v(t)(x)-\bar u(t)(x))\right)\theta(t)(x)\,dx.
\end{multline*}
Therefore, by Proposition \ref{prop2}
\begin{equation}
\frac d{dt}\frac12|\theta(t)|_{L^2}^2+\lambda_0|\theta(t)|_{H^1}\leq I_1(t)+I_2(t)+I_3(t),
\end{equation}
where
\begin{equation}
I_1(t):=\int_\Omega \partial_uf(x,0)|\theta(t)(x)|^2\,dx,
\end{equation}
\begin{equation}
I_1(t):=\int_\Omega( \partial_uf(x,\bar u(t)(x))-\partial_uf(x,0))|\theta(t)(x)|^2\,dx
\end{equation}
and
\begin{equation}
I_3(t)=\int_\Omega \left(f(x,\bar v(t))-f(x,\bar u(t))-\partial_uf(x,\bar u(t))(\bar v(t)-\bar u(t))\right)\theta(t)\,dx.
\end{equation}
Repeating the same computations of the proof of Proposition
\ref{unif-diff-1}, for $\epsilon>0$ we get
\begin{equation}
I_1(t)+I_2(t)\leq \epsilon
|\theta(t)|_{H^1}^2+C_1(\epsilon,|\mathcal
I|_{H^1})|\theta(t)|_{L^2}^2.
\end{equation}
Concerning $I_3(t)$, for $\epsilon>0$ we have
\begin{multline*}
I_2(t)\leq\int_\Omega C(1+|\bar u(t)(x)|^\gamma+|\bar
v(t)(x)|^\gamma)|\bar v(t)(x)-\bar u(t)(x)|^2\theta(t)(x) \,dx\\
\leq C |\theta(t)|_{L^6}|\bar v(t)-\bar u(t)|_{L^{12/5}}^2+
C|\theta(t)|_{L^6}(|\bar u(t)|_{L^6}^\gamma+|\bar
v(t)|_{L^6}^\gamma)|\bar v(t)-\bar u(t)|^2_{L^{12/(5-\gamma)}}\\
\leq \epsilon|\theta(t)|_{H^1}^2+ C_2(\epsilon,|\mathcal
I|_{H^1})(|\bar v(t)-\bar u(t)|_{L^{12/5}}^4+|\bar v(t)-\bar
u(t)|^4_{L^{12/(5-\gamma)}})\\ \leq\epsilon|\theta(t)|_{H^1}^2+
C_3(\epsilon,|\mathcal I|_{H^1})(|\bar v(t)-\bar u(t)|_{H^1}|\bar
v(t)-\bar u(t)|_{L^2}^3+|\bar v(t)-\bar
u(t)|_{H^1}^{1+\gamma}|\bar v(t)-\bar u(t)|_{L^2}^{3-\gamma})\\
\end{multline*}
Choosing $\epsilon=\lambda_0/2$, we get
\begin{multline*}
\frac d{dt}\frac12|\theta(t)|_{L^2}^2-C_1(\epsilon,|\mathcal
I|_{H^1})|\theta(t)|_{L^2}^2\\ \leq C_3(\epsilon,|\mathcal
I|_{H^1})(|\bar v(t)-\bar u(t)|_{H^1}|\bar v(t)-\bar
u(t)|_{L^2}^3+|\bar v(t)-\bar u(t)|_{H^1}^{1+\gamma}|\bar
v(t)-\bar u(t)|_{L^2}^{3-\gamma})\\ \leq C_4(\epsilon,|\mathcal
I|_{H^1})(|\bar v(t)-\bar u(t)|_{L^2}^3+|\bar v(t)-\bar
u(t)|_{H^1}^{2}|\bar v(t)-\bar u(t)|_{L^2}^{3-\gamma}).
\end{multline*}
By Lemma \ref{cont-dip}, we get
\begin{multline*}
\frac d{dt}\frac12|\theta(t)|_{L^2}^2-C_1(\epsilon,|\mathcal
I|_{H^1})|\theta(t)|_{L^2}^2\\ \leq C_4(\epsilon,|\mathcal
I|_{H^1})(e^{3Kt}|\bar v_0-\bar u_0|_{L^2}^3+e^{(3-\gamma)Kt}|\bar
v(t)-\bar u(t)|_{H^1}^{2} |\bar v_0-\bar u_0|_{L^2}^{3-\gamma}).
\end{multline*}
Writing $C_1$ for $C_1(\epsilon,|\mathcal I|_{H^1})$ and $C_4$ for
$C_4(\epsilon,|\mathcal I|_{H^1})$, we have
\begin{multline}\label{ming}
\frac d{dt}\frac12(e^{-C_1 t}|\theta(t)|_{L^2}^2)\\ \leq
C_4(e^{(3K-C_1)t}|\bar v_0-\bar
u_0|_{L^2}^3+e^{((3-\gamma)K-C_1)t}|\bar v(t)-\bar u(t)|_{H^1}^{2}
|\bar v_0-\bar u_0|_{L^2}^{3-\gamma}).
\end{multline}
Finally, integrating (\ref{ming}), recalling that $\theta(0)=0$
and taking into account Lemma \ref{cont-dip}, we get the existence
of two increasing functions $\Phi_1(t)$ and $\Phi_2(t)$ such that
\begin{equation*}
|\theta(t)|^2_{L^2}\leq \Phi_1(t)|\bar v_0-\bar u_0|_{L^2}^3
+\Phi_2(t)|\bar v_0-\bar u_0|_{L^2}^{5-\gamma},
\end{equation*}
and the thesis follows.
\end{proof}

%----------------------------------------------------------

\section{Dimension of invariant sets}
Let $\mathcal X$ be a complete metric space and let $\mathcal
K\subset\mathcal X$ be a compact set. For $d\in \R^+$ and
$\epsilon>0$ one defines
\begin{equation}
\mu_{H}(\mathcal K,d,\epsilon):=\inf \left\{\sum_{i\in I}r_i^d\mid
\mathcal K\subset \bigcup_{i\in I}B(x_i,r_i),\, r_i\leq\epsilon
\right\},
\end{equation}
where the infimum is taken over all the finite coverings of
$\mathcal K$ with balls of radius $r_i\leq\epsilon$. Observe that
$\mu_{H}(\mathcal K,d,\epsilon)$ is a non increasing function of
$\epsilon$ and $d$. The $d$-dimensional Hausdorff measure of
$\mathcal K$ is by definition
\begin{equation}
\mu_H(\mathcal K,d):=\lim_{\epsilon\to 0}\mu_{H}(\mathcal
K,d,\epsilon)= \sup_{\epsilon>0}\mu_{H}(\mathcal K,d,\epsilon).
\end{equation}
One has:
\begin{enumerate}
\item $\mu_H(\mathcal K,d)\in[0,+\infty]$;
\item if $\mu_H(\mathcal K,\bar d)<\infty$, then $\mu_H(\mathcal
K,d)=0$ for all $d>\bar d$; \item if $\mu_H(\mathcal K,\bar d)>0$,
then $\mu_H(\mathcal K,d)=+\infty$ for all $d<\bar d$.
\end{enumerate}
The Hausdorff dimension of $\mathcal K$ is the smallest $d$ for
which $\mu_H(\mathcal K,d)$ is finite, i.e.
\begin{equation}
{\rm dim}_{H}(\mathcal K):=\inf\{d>0\mid \mu_H(\mathcal K,d)=0\}.
\end{equation}
As pointed up in \cite{Sch}, the Hausdorff dimension is in fact an
intrinsic metric property of the set $\mathcal K$. Moreover, if
$\mathcal Y$ is another complete metric space and $\ell\colon
\mathcal K\to\mathcal Y$ is a Lipschitzian map, then ${\rm dim
}_H(\ell(\mathcal K))\leq {\rm dim}_H(\mathcal K)$.

There is a well developed technique to estimate the Hausdorff
dimension of an invariant set of a map or a semigroup. We refer
the reader e.g. to \cite{Tem} and \cite{Lady}. The geometric idea
consists in tracking the evolution of a $d$-dimensional volume
under the action of the linearization of the semigroup along
solutions lying in the invariant set. One looks then for the
smallest $d$ for which any $d$-dimensional volume contracts
asymptotically as $t\to\infty$.

Let $\bar u_0\in\mathcal I$ and let $\bar u(\cdot)\colon\R\to
H^1_0(\Omega)$ be a full bounded solution of (\ref{equation2})
such that $\bar u(0)=\bar u_0$ and $\bar u(t)\in\mathcal I$ for
$t\in\R$. For $t\geq 0$, we denote by $a_{\bar u_0 }(t;u,v)$ the
bilinear form defined by (\ref{non-aut-form}), and by $A_{\bar
u_0}(t)$  the self-adjoint operator determined by the relation
(\ref{non-aut-op}). Given a $d$-dimensional subspace $E_d$ of
$L^2(\Omega)$, with $E_d\subset H^1_0(\Omega)$, we define the
operator $A_{\bar u_0}(t\mid E_d)\colon E_d\to E_d$ by
\begin{equation}
\langle A_{\bar u_0}(t\mid E_d) \phi,\psi\rangle_{L^2}:=a_{\bar
u_0 }(t;\phi,\psi),\quad\phi,\psi\in E_d.
\end{equation}
Notice that, if $E_d\subset D(A_{\bar u_0}(t))$, then one has
$A_{\bar u_0}(t\mid E_d)=P_{E_d}A_{\bar u_0 }(t)P_{E_d}|_{E_d}$,
where $P_{E_d}\colon L^2(\Omega)\to E_d$ is the $L^2$-orthogonal
projection onto $E_d$. We define
\begin{equation}
{\rm Tr}_d(A_{\bar u_0}(t)):=\inf_{\substack{E_d\subset
H^1_0(\Omega)\\{\rm dim}\, E_d=d}}{\rm Tr}(A_{\bar u_0}(t\mid
E_d)).
\end{equation}

Let $\bar u_0\in\mathcal I$, let $d\in\N$ and let $v_{0,i}\in
L^2(\Omega)$, $i=1$, \dots, $d$. Set $v_i(t):=\mathcal U(\bar
u_0;t)v_{0,i}$, $t\geq 0$, where $\mathcal U(\bar u_0;t)$ is
defined by (\ref{def-diff}). We denote by $G(t)$ the
$d$-dimensional volume delimited by $v_1(t)$, \dots, $v_d(t)$ in
$L^2(\Omega)$, that is
\begin{equation}
G(t):=|v_1(t)\wedge v_2(t)\wedge\cdots\wedge v_d(t)|_{\wedge^d L^2}=\left(\det(\langle v_i(t),v_j(t)\rangle_{L^2})_{ij}\right)^{1/2}.
\end{equation}
An easy computation using Leibnitz rule and Proposition \ref{sol-oper} shows that, for $t>0$, $G(t)$ satisfies the ordinary differential equation
\begin{equation}
 G'(t)=-{\rm Tr}(A_{\bar u_0}(t\mid E_d(t))G(t),
\end{equation}
where $E_d(t):={\rm span}(v_1(t),\dots,v_d(t))$. It follows from
Propositions \ref{unif-diff-1} and  \ref{unif-diff-2} and from the results in \cite[Ch. V]{Tem}
that the Hausdorff dimension ${\rm dim}_{H}(\mathcal I)$ of
$\mathcal I$ in $L^2(\Omega)$ is finite and less than or equal to $d$, provided
\begin{equation}
\limsup_{t\to\infty}\sup_{\bar u_0\in\mathcal I}\frac1t\int_0^t-{\rm Tr}_d(A_{\bar u_0}(s))\,ds<0.
\end{equation}
Therefore, in order to prove that ${\rm dim}_{H}(\mathcal I)\leq
d$, we are lead to estimate $-{\rm Tr}_d(A_{\bar u_0}(t))$. To
this end, we notice that, whenever $E_d$ is a $d$-dimensional
subspace of $L^2(\Omega)$,  and $B\colon E_d\to E_d$ is a
selfadjoint operator, then
\begin{equation*}
{\rm Tr}(B)=\sum_{i=1}^d\langle B\phi_i,\phi_i\rangle_{L^2},
\end{equation*}
where $\phi_1$, \dots, $\phi_d$ is any $L^2$-orthonormal basis of
$E_d$. So let $E_d\subset H^1_0(\Omega)$ be a $d$-dimensional
space and let $\phi_1$, \dots, $\phi_d$ be an $L^2$-orthonormal
basis of $E_d$. Fix $0<\delta<1$. It follows that
\begin{multline}
{\rm Tr}(A_{\bar u_0}(t\mid E_d))\\
=\sum_{i=1}^d\left((1-\delta)\left(\int_\Omega|\nabla\phi_i|^2\,dx+\int_\Omega\beta(x)|\phi_i|^2\,dx\right)-\int_\Omega\partial_u
f(x,0)|\phi_i|^2\,dx\right)\\
+\delta\sum_{i=1}^d\left(\int_\Omega|\nabla\phi_i|^2\,dx+\int_\Omega\beta(x)|\phi_i|^2\,dx\right)\\
+\sum_{i=1}^d\int_\Omega\left(\partial_u f(x,\bar u(t))-\partial_u
f(x,0)\right)|\phi_i|^2\,dx.
\end{multline}

We introduce the following bilinear form defined on the space
$H^1_0(\Omega)$:
\begin{multline}
a_\delta(u,v):=(1-\delta)\left(\int_\Omega\nabla u(x)\cdot\nabla
v(x)\,dx+\int_\Omega\beta(x)u(x)v(x)\,dx\right)\\- \int_\Omega
\partial_uf(x,0)u(x)v(x)\,dx,\quad u,v\in H^1_0(\Omega).
\end{multline}
Let $A_\delta$ be the self-adjoint operator determined by the
relation
\begin{equation}
\langle A_\delta u,v\rangle_{L^2}=a_\delta(u,v),\quad u\in
D(A_\delta),v\in H^1_0(\Omega).
\end{equation}
Given a $d$-dimensional subspace $E_d$ of $L^2(\Omega)$, with
$E_d\subset H^1_0(\Omega)$, we define the operator
$A_{\delta}(E_d)\colon E_d\to E_d$ by
\begin{equation}
\langle A_{\delta}( E_d)
\phi,\psi\rangle_{L^2}:=a_{\delta}(\phi,\psi),\quad\phi,\psi\in
E_d.
\end{equation}

It follows that
\begin{multline}
{\rm Tr}(A_{\bar u_0}(t\mid E_d))={\rm Tr}(A_{\delta}(E_d))
+\delta\sum_{i=1}^d\left(\int_\Omega|\nabla\phi_i|^2\,dx+\int_\Omega\beta(x)|\phi_i|^2\,dx\right)\\
+\sum_{i=1}^d\int_\Omega\left(\partial_u f(x,\bar u(t))-\partial_u
f(x,0)\right)|\phi_i|^2\,dx.
\end{multline}

We introduce the \emph{proper values} of the operator $A_\delta$:
\begin{equation}
\mu_j(A_\delta):=\sup_{\psi_1,\dots,\psi_{j-1}\in H^1_0(\Omega)}
\inf_{\substack{\psi\in[\psi_1,\dots,\psi_{j-1}]^\perp\\|\psi|_{L^2}=1,\,\psi\in
H^1_0(\Omega) }} a_\delta(\psi,\psi)\quad j=1,2,\dots
\end{equation}

We recall (see e.g.  Theorem XIII.1 in \cite{RS}) that:
\begin{Prop}\label{min-max}
For each fixed $n$, either
\begin{itemize}
\item[{\it 1)}] there are at least $n$ eigenvalues (counting multiplicity) below the bottom of the essential
 spectrum of $A_\delta$ and $\mu_n(A_\delta)$ is the $n$th eigenvalue (counting multiplicity);
\end{itemize}
or
\begin{itemize}
\item[{\it 2)}] $\mu_n(A_\delta)$ is the bottom of the essential spectrum and in that
case $\mu_{n+j}(A_\delta)=\mu_n(A_\delta)$, $j=1,2,\dots$ and there are at most $n-1$ eigenvalues
(counting multiplicity) below $\mu_n(A_\delta)$.\qed
\end{itemize}
\end{Prop}

Let $\mu_j(A_\delta(E_d))$, $j=1$, \dots, $d$, be the
\emph{eigenvalues} of  $A_\delta(E_d)$. By Theorem XIII.3 in
\cite{RS}, we have that
\begin{equation}
\mu_j(A_\delta(E_d))\geq \mu_j(A_\delta),\quad j=1,\dots,d.
\end{equation}
It follows that
\begin{multline}
{\rm Tr}(A_{\bar u_0}(t\mid E_d))\geq\sum_{i=1}^d\mu_i(A_\delta)
+\delta\sum_{i=1}^d\left(\int_\Omega|\nabla\phi_i|^2\,dx+\int_\Omega\beta(x)|\phi_i|^2\,dx\right)\\
+\sum_{i=1}^d\int_\Omega\left(\partial_u f(x,\bar u(t))-\partial_u
f(x,0)\right)|\phi_i|^2\,dx.
\end{multline}

To proceed further, we need to recall the Lieb-Thirring inequality
(see \cite{LT}).
\begin{Prop}\label{lieb-thirr-th}
Let $N\in\N$ and let  $p\in \R$, with $\max\{N/2,1\}\leq p\leq 1+N/2$. There exists a constant
$K_{p,N}>0$ such that, if $\phi_1$, \dots, $\phi_d\in H^1(\R^N)$ are
pairwise $L^2$-orthonormal, then
\begin{equation}\label{lieb-thirr-ineq}
\sum_{i=1}^d\int_{\R^N}|\nabla\phi_i(x)|^2\,dx\geq
\frac1{K_{p,N}}\left(\int_{\R^N}\rho(x)^{p/(p-1)}\,dx\right)^{2(p-1)/N},
\end{equation}
where $\rho(x):=\sum_{i=1}^d|\phi_i(x)|^2$.\qed
\end{Prop}

Now we have:
\begin{Lemma}\label{piombo}
Let $\bar u\in\mathcal I$ and let $\phi_1$, \dots, $\phi_d\in H^1(\R^N)$ be
pairwise $L^2$-orthonormal. Then
\begin{multline}
\delta\sum_{i=1}^d\left(\int_\Omega|\nabla\phi_i|^2\,dx+\int_\Omega\beta(x)|\phi_i|^2\,dx\right)\\
+\sum_{i=1}^d\int_\Omega\left(\partial_u f(x,\bar u(x))-\partial_u
f(x,0)\right)|\phi_i|^2\,dx\geq -D(\gamma,\lambda_0,\delta,|\mathcal I|_{H^1}),
\end{multline}
where
\begin{multline}
D(\gamma,\lambda_0,\delta,|\mathcal I|_{H^1})\\
=\frac52\left(\frac35\frac2{\delta\lambda_0}\right)^{\frac32}\left(C|\mathcal I|_{L^{5/2}}K_{5/2,3}\right)^{\frac52}
+\frac{3-\gamma}4\left(\frac{\gamma+1}4\frac2{\delta\lambda_0}\right)^{\frac{\gamma+1}{3-\gamma}}
\left(C|\mathcal I|_{L^{6}}^{\gamma+1}K_{6/(\gamma+1),3}\right)^{\frac4{3-\gamma}}.
\end{multline}
\end{Lemma}
\begin{proof}
We observe first that
\begin{equation}
\delta\sum_{i=1}^d\left(\int_\Omega|\nabla\phi_i|^2\,dx+\int_\Omega\beta(x)|\phi_i|^2\,dx\right)
\geq\delta\lambda_0\sum_{i=1}^d\int_{\Omega}|\nabla\phi_i|^2\,dx.
\end{equation}
On the other hand,
\begin{multline}
\left| \int_\Omega\left(\partial_u f(x,\bar u(x))-\partial_u
f(x,0)\right)\rho(x)\,dx\right|\leq\int_\Omega C(1+|\bar u|^\gamma)|\bar u||\rho|\,dx\\
\leq C|\bar u|_{L^{5/2}}|\rho|_{L^{5/3}}+C|\bar u|_{L^{6}}^{\gamma+1}|\rho|_{L^{6/(5-\gamma)}}.
\end{multline}
By Lieb-Thirring inequality (\ref{lieb-thirr-ineq}), we have
\begin{multline}
\left| \int_\Omega\left(\partial_u f(x,\bar u(x))-\partial_u
f(x,0)\right)\rho(x)\,dx\right|\\
\leq C|\mathcal I|_{L^{5/2}}K_{5/2,3}\left(\sum_{i=1}^d\int_{\R^N}|\nabla\phi_i|^2\,dx\right)^{3/5}\\
+C|\mathcal I|_{L^{6}}^{\gamma+1}K_{6/(\gamma+1),3}\left(\sum_{i=1}^d\int_{\R^N}|\nabla\phi_i|^2\,dx\right)^{(\gamma+1)/4}.
\end{multline}
The conclusion follows by a simple application of Young's inequality.
\end{proof}
Thanks to Lemma \ref{piombo}, we finally get:
\begin{equation}
{\rm Tr}(A_{\bar u_0}(t\mid E_d))\geq\sum_{i=1}^d\mu_i(A_\delta)
-D(\gamma,\lambda_0,\delta,|\mathcal I|_{H^1}).
\end{equation}
Therefore, in order to conclude that ${\rm dim}_{H}(\mathcal I)$
is finite, we are lead to make some assumption which guarantees
that $\sum_{i=1}^d\mu_i(A_\delta)$ can be made positive and as
large as we want, by choosing $d$ is sufficiently large. This is
equivalent to the fact that the bottom of the essential spectrum
of $A_\delta$ be strictly positive. We make the following
assumption:
\begin{Hyp}\label{hyp3}
For every $\epsilon>0$ there exists $V_\epsilon\in L^{r}(\Omega)$, $r> 3/2$, $V_\epsilon\geq0$, such
that $\partial_u f(x,0)\leq V_\epsilon(x)+\epsilon$, for $x\in\Omega$.
\end{Hyp}

We need the following lemmas:

\begin{Lemma}\label{mandrake1}
Let  $r>3/2$ and let $V\in L^{r}(\Omega)$. If $r>3$ let $p:=2$; if $r\leq3$ let $p:=6/5$. Then the assignment $u\mapsto V u$ defines a
compact map from $H^1_0(\Omega)$ to $L^p(\Omega)$, and hence to $H^{-1}(\Omega)$.
\end{Lemma}
\begin{proof}
Let $B\subset H^1_0(\Omega)$ be bounded. If $\mathcal B$ is a
Banach space such that $H^1_0(\Omega)\subset \mathcal B$, we
define $|B|_{\mathcal B}:=\sup \{|u|_{\mathcal B}\mid u\in B\}$.
If $u\in H^1_0(\Omega)$ we denote by $\tilde u$ its trivial
extension to the whole $\R^3$. Similarly, we denote by $\tilde V$
the trivial extension of $V$ to $\R^3$. For $k>0$, let $\chi_k$ be
the characteristic function of the set $\{x\in\R^3\mid |x|\leq
k\}$. Now, for $u\in B$ and $k>0$, we have:
\begin{equation}
\int_{\R^3}|(1-\chi_k)\tilde V\tilde
u|^p\,dx\leq\left(\int_{|x|\geq k}|\tilde
V|^r\,dx\right)^{p/r}\left(\int_{|x|\geq k}|\tilde
u|^{\frac{pr}{r-p}}\,dx\right)^{(r-p)/r}.
\end{equation}
It follows that
\begin{equation}
|(1-\chi_k)\tilde V\tilde
u|_{L^p}\leq|B|_{L^{\frac{pr}{r-p}}}|(1-\chi_k)\tilde V|_{L^r},
\quad u\in B, k>0.
\end{equation}
Similarly, we have:
\begin{equation}\label{spritz}
|\chi_k\tilde V\tilde u|_{L^p}\leq|\tilde V|_{L^r}|\chi_k \tilde u
|_{L^{\frac{pr}{r-p}}}, \quad u\in H^1_0(\Omega), k>0.
\end{equation}
Now, given $\epsilon>0$, we choose $k>0$ so large that
$|(1-\chi_k)\tilde V|_{L^r}\leq\epsilon$. Then
\begin{multline}
\{\tilde V\tilde u\mid u\in B\}=\{\chi_k\tilde V\tilde
u+(1-\chi_k)\tilde V\tilde u\mid u\in B\}\\ \subset \{\chi_k\tilde
V\tilde u\mid u\in B\}+\{(1-\chi_k)\tilde V\tilde u\mid u\in B\}\\
\subset \{v\in L^p(\R^3)\mid |v|_{L^p}\leq\epsilon\}+
\{\chi_k\tilde V\tilde u\mid u\in B\}.
\end{multline}
We notice that $2\leq pr/(r-p)<6$: therefore, By Rellich's
Theorem, $H^1(B_k(0))$ is compactly embedded into
${L^{\frac{pr}{r-p}}}$. It follows that the set $\{\chi_k\tilde
u\mid u\in B\}$ is precompact in ${L^{\frac{pr}{r-p}}}$. By
(\ref{spritz}), we deduce that $\{\chi_k\tilde V\tilde u\mid u\in
B\}$ is precompact in $L^p(\R^3)$. A simple \emph{measure of non
compactness argument} shows then that the set $\{\tilde V\tilde
u\mid u\in B\}$ is precompact in $L^p(\R^3)$ and this in turn
implies that the set $\{ V u\mid u\in B\}$ is precompact in
$L^p(\Omega)$.
\end{proof}

\begin{Lemma}\label{mandrake2}
Let $V$ be as in Lemma \ref{mandrake1}. Let $A+V$ be the selfadjoint operator determined by
the bilinear form $a(u,v)+\int_\Omega V uv\,dx$, $u,v\in H^1_0(\Omega)$. Then, for sufficiently large
$\lambda>0$, $(A+\lambda)^{-1}-(A+V+\lambda)^{-1}$ is a compact operator in $L^2(\Omega)$.
\end{Lemma}
\begin{proof}
Take $\lambda>0$ so large that $A+V+\lambda$ be strictly positive.
Let $u\in L^2(\Omega)$. Set $v:=(A+V+\lambda)^{-1}u$,
$w:=(A+\lambda)^{-1}u$ and $z:=v-w$. This means that
\begin{equation}
a(v,\phi)+\lambda(v,\phi)+\int_\Omega V v\phi\,dx=\int_\Omega u\phi\,dx, \quad \text{for all $\phi\in H^1_0(\Omega)$}
\end{equation}
and
\begin{equation}
a(w,\phi)+\lambda(w,\phi)=\int_\Omega u\phi\,dx, \quad \text{for all $\phi\in H^1_0(\Omega)$}.
\end{equation}
It follows that
\begin{equation}
a(z,\phi)+\lambda(z,\phi)+\int_\Omega V v\phi\,dx=0, \quad \text{for all $\phi\in H^1_0(\Omega)$}.
\end{equation}
Choosing $\phi:=z$, Proposition \ref{prop2} and Lemma \ref{mandrake1} imply
\begin{equation}
\lambda_0 |z|_{H^1}^2\leq |z|_{H^1}|Vv|_{H^-1}\leq
\frac{\lambda_0}2 |z|_{H^1}^2+ K_{\lambda_0}|Vv|_{H^-1}^2.
\end{equation}
Therefore we obtain the estimate
\begin{equation}
|(A+\lambda)^{-1}u-(A+V+\lambda)^{-1}u|_{H^1}\leq K_{\lambda_0}|V(A+V+\lambda)^{-1}u|_{H^{-1}}, \quad u\in L^2(\Omega),
\end{equation}
and the conclusion follows from Lemma \ref{mandrake1}.
\end{proof}

Now we can prove:

\begin{Prop} Assume Hypothesis \ref{hyp3} is satisfied. Then the essential spectrum of $A_\delta$ is contained in $[(1-\delta)\lambda_1,+\infty[$.
\end{Prop}
\begin{proof}
Hypothesis \ref{hyp3} and Proposition \ref{min-max}
imply that, for every $\epsilon>0$, the bottom of the essential spectrum of
$A_\delta$ is larger than or equal to the bottom of the essential spectrum of $(1-\delta)A-\epsilon-V_\epsilon(x)$.
We observe that the spectrum of $(1-\delta)A-\epsilon$ is contained in $[(1-\delta)\lambda_1-\epsilon,+\infty[$.
By Lemma \ref{mandrake2} and Weyl's Theorem (see \cite[Theorem XIII.14]{RS}), the essential spectrum
of $(1-\delta)A-\epsilon-V_\epsilon(x)$ coincides with that of $(1-\delta)A-\epsilon$. It follows that
 the bottom of the essential spectrum of $A_\delta$ is larger than or equal to $(1-\delta)\lambda_1-\epsilon$
 for arbitrary small $\epsilon>0$, and the conclusion follows.
\end{proof}

Whenever Hypothesis \ref{hyp3} is satisfied, for $0<\delta<1$ and
$\lambda<(1-\delta)\lambda_0$ we introduce the following quantity:
\begin{equation}
\mathcal N(\delta,\lambda):=\#\text{ eigenvalues of $A_\delta$
below $\lambda$}
\end{equation}
Then, for $d\geq \mathcal N\left(\delta,\frac{1-\delta}2 \lambda_1
\right)$ we have:
\begin{equation}
\sum_{i=1}^d\mu_i(A_\delta)\geq \mathcal
N\left(\delta,\frac{1-\delta}2 \lambda_1
\right)\mu_1(\delta)+\left(d-\mathcal
N\left(\delta,\frac{1-\delta}2 \lambda_1
\right)\right)\frac{1-\delta}2 \lambda_1
\end{equation}
We have thus proved our first main result:
\begin{Theorem}\label{dimension} Assume Hypotheses \ref{hyp1}, \ref{hyp2} and
\ref{hyp3} are satisfied. Let $\mathcal I\subset H^1_0(\Omega)$ be
a compact invariant set for the semiflow $\pi$ generated by
equation (\ref{equation2}) in $H^1_0(\Omega)$. Then the Hausdorff
dimension of $\mathcal I$ in $L^2(\Omega)$ is finite and less than
or equal to $d$, provided $d$ is an integer number larger than
$\max\{d_1,d_2\}$, where
\begin{equation}
d_1:=\mathcal N\left(\delta,\frac{1-\delta}2 \lambda_1 \right)
\end{equation}
and
\begin{equation}
d_2:=\frac{2}{(1-\delta)\lambda_1}\left(\mathcal
N\left(\delta,\frac{1-\delta}2 \lambda_1
\right)\left(\frac{1-\delta}2
\lambda_1-\mu_1(A_\delta)\right)+D(\gamma,\lambda_0,\delta,|\mathcal
I |_{H^1})\right).
\end{equation}\qed
\end{Theorem}

\begin{Rem} The first proper value $\mu_1(A_\delta)$ of
$A_\delta$ can be estimated from below in terms of $\lambda_0$ and
$|\partial_u f(\cdot,0)|_{L^\sigma_{\rm u}}$. The explicit
computations are left to the reader.
\end{Rem}

\begin{Rem}
By Lemma \ref{liplip}, also  the Hausdorff dimension
of $\mathcal I$ in $H^1_0(\Omega)$ is finite and it is equal to the Hausdorff dimension
of $\mathcal I$ in $L^2(\Omega)$.
\end{Rem}

%-----------------------------------------------------------------

\section{Estimate of $\mathcal
N\left(\delta,\frac{1-\delta}2 \lambda_1 \right)$}
In this section
we shall obtain an explicit estimate for the number $\mathcal
N\left(\delta,\frac{1-\delta}2 \lambda_1 \right)$ in terms of the
dominating potential $V_\epsilon$ of Hypothesis \ref{hyp3}. Our
main tool is the celebrated Cwickel-Lieb-Rozenblum inequality, in
its abstract formulation due to Rozenblum and Solomyak (see
\cite{RoSo}). In order to exploit the CLR inequality, we need to
make some assumption on the regularity of the open domain
$\Omega$. Namely, we make the following assumption:

\begin{Hyp}\label{hyp4}
The open set  $\Omega$ is a uniformly $C^2$
domain in the sense of Browder \cite[p. 36]{Brow}.
\end{Hyp}

As a consequence,  by elliptic
regularity we have that $D(-\Delta)=H^2(\Omega)\cap
H^1_0(\Omega)\subset L^\infty(\Omega)$. In this situation, if
$\omega\in L^\sigma_{\rm u}(\R^3)$ then the assignment $u\mapsto
\omega u$ defines a relatively bounded perturbation of $-\Delta$
and therefore $D(-\Delta+\omega)=H^2(\Omega)\cap H^1_0(\Omega)$.
It follows that $X^\alpha\subset L^\infty(\Omega)$ for
$\alpha>3/4$ (see \cite[Th. 1.6.1]{He}).

Set $\bar\epsilon:=(1-\delta)\lambda_1/4$. Define the bilinear
forms
\begin{multline}
\tilde
a_{\delta,\bar\epsilon}(u,v):=(1-\delta)\left(\int_\Omega\nabla
u\cdot\nabla v \,dx+\int_\Omega\beta
uv\,dx\right)-3\bar\epsilon\int_\Omega uv\,dx,\\ u,v\in
H^1_0(\Omega),
\end{multline}
and
\begin{equation}
b_{\delta,\bar\epsilon}(u,v):=-\int_\Omega V_{\bar\epsilon}
uv\,dx.
\end{equation}
Moreover, set
\begin{equation} a_{\delta,\bar\epsilon}(u,v):=\tilde
a_{\delta,\bar\epsilon}(u,v)+b_{\delta,\bar\epsilon}(u,v)
\end{equation}
and denote by $\tilde A_{\delta,\bar\epsilon}$ and
$A_{\delta,\bar\epsilon}$ the selfadjoint operators induced by
$\tilde a_{\delta,\bar\epsilon}$ and $a_{\delta,\bar\epsilon}$
respectively.

A simple computation shows that
\begin{equation}
\mathcal N\left(\delta,\frac{1-\delta}2 \lambda_1 \right)\leq
n_{\delta,\bar\epsilon},
\end{equation}
where $n_{\delta,\bar\epsilon}$ is the number of negative
eigenvalues of $A_{\delta,\bar\epsilon}$.

By Theorem 1.3.2 in \cite{Da}, the operator $\tilde A_{\delta,\bar
\epsilon}$ is positive (with $\tilde A_{\delta,\bar \epsilon}\geq
\bar\epsilon I$) and order preserving. Moreover, since
$D(A_{\delta,\bar \epsilon}^\alpha)\subset L^\infty(\Omega)$ for $
\alpha>3/4$, then for every such $\alpha$ and
$\gamma<\bar\epsilon$ we have
\begin{equation}
|e^{-t\tilde A_{\delta,\bar \epsilon}}u|_{L^\infty}\leq M_{\alpha,\gamma}t^{-\alpha}e^{-\gamma t}|u|_{L^2}, \quad u\in L^2(\Omega),
\end{equation}
where $M_{\alpha,\gamma}$ is a constant depending only on
$\alpha$, $\gamma$ and on the embedding constant of $H^2(\Omega)$
into $L^\infty(\Omega)$. It follows that
\begin{equation}
M_{{\tilde A_{\delta,\bar \epsilon}}}(t):=\|e^{-(t/2)\tilde A_{\delta,\bar \epsilon}}\|_{\mathcal L(L^2,L^\infty)}^2
\leq M_{\alpha,\gamma}^2 2^{2\alpha}t^{-2\alpha}e^{-\gamma t}.
\end{equation}
We are now in a position to apply Theorem 2.1 in \cite{RoSo}. We have thus proved the following theorem:
\begin{Theorem}\label{clr}
Assume that Hypotheses \ref{hyp1}, \ref{hyp2}, \ref{hyp3} and \ref{hyp4} are satisfied. Let $\bar\epsilon:=(1-\delta)\lambda_1/4$. Then
\begin{equation}
\mathcal N\left(\delta,\frac{1-\delta}2 \lambda_1 \right)\leq
n_{\delta,\bar\epsilon}\leq C_{2q} M_{2q,\gamma}\int_\Omega V_{\bar\epsilon}(x)^q \,dx,
\end{equation}
where $C_\alpha$ is a constant depending only on $\alpha$, for $\alpha>3/4$.\qed
\end{Theorem}

%--------------------------------------------------------------------------

\section{Dissipative equations: dimension of the attractor}

In this section we specialize our results to the case of a
dissipative equation.
 We make the following assumption:

\begin{Hyp}\label{hyp5}
There exists a non negative function $D\in L^q(\Omega)$, $2\geq q>3/2$, such that
\begin{equation}\label{dissip}
f(x,u)u\leq D(x)|u|,\quad (x,u)\in \Omega\times\R.
\end{equation}
\end{Hyp}

\begin{Rem}  Hypotheses \ref{hyp5} and \ref{hyp1} together are equivalent to the structure assumption of Theorem 4.4 in
\cite{AMR}.\end{Rem} An easy computation shows that $|f(x,0)|\leq
D(x)$  for $x\in\Omega$, and that
 $F(x,u):=\int_0^uf(x,s)\,ds$
satisfies
\begin{equation*}
F(x,u)\leq D(x)|u|, \quad  (x,u)\in \Omega\times\R.
\end{equation*}
By slightly modifying some technical arguments in \cite{PR1}, one
can prove that the semiflow $\pi$ generated by equation
(\ref{equation2}) in $H^1_0(\Omega)$ possesses a compact global
attractor $\mathcal A$. Moreover, $\pi$ is gradient-like with
respect to the Lyapunov functional
\begin{equation}
\mathcal L(u):=\int_\Omega|\nabla
u|^2\,dx+\int_\Omega\beta(x)|u|^2\,dx-\int_\Omega F(x,u)\,dx,\quad
u\in H^1(\Omega).
\end{equation}

Assuming Hypothesis \ref{hyp5}, we shall give an explicit estimate
for $|\mathcal A|_{H^1}$ in terms of $|D|_{L^q}$. Moreover, we
shall prove that Hypothesis \ref{hyp5} implies  Hypothesis \ref{hyp3}, and we
explicitly compute the dominating potential $V_\epsilon$ in terms
of $D$. Therefore, we are able to obtain an explicit estimate for
the number $\mathcal N\left(\delta,\frac{1-\delta}2 \lambda_1
\right)$ in terms of $|D|_{L^q}$. As a consequence, the estimate
of the dimension of $\mathcal A$ given by Theorem \ref{dimension}
can be made completely explicit in terms of the structure
parameters of  equation (\ref{equation1}).

We have the following theorem:

\begin{Theorem}\label{special2} Assume Hypotheses \ref{hyp1}, \ref{hyp2} and
\ref{hyp5} are satisfied.
\begin{enumerate}
\item Let $\phi\in H^1_0(\Omega)$ be an equilibrium of $\pi$. Then
\begin{equation*}
|\phi|_{H^1}\leq \frac{M_{q'}}{\lambda_0} |D|_{L^q},
\end{equation*}
where $M_{q'}$ is the embedding constant of $H^1_0(\R^3)$ into
$L^{q'}(\R^3)$.
\item There exists a constant $S>0$ such that
\begin{equation*}
|u|_{H^1}\leq S\quad\text{for all $u\in\mathcal A$};
\end{equation*}
The constant $S$ can be explicitly computed and depends only on
$C$, $\gamma$, $\sigma$, $\lambda_0$, $\Lambda_0$, $|D|_{L^q}$,
$|\partial_u f(\cdot,o)|_{L^\sigma_{\rm u}}$ and on the constants
of Sobolev embeddings.
\end{enumerate}
\end{Theorem}
\begin{proof}
Let $\phi\in H^1_0(\Omega)$ be an equilibrium of $\pi$. Then, for $\epsilon>0$, we have
\begin{multline*}
\lambda_0|\phi|_{H^1}^2\leq\int_\Omega|\nabla \phi|^2\,dx+\int_\Omega\beta(x)|\phi|^2\,dx=\int_\Omega f(x,\phi)\phi\,dx\leq\int_\Omega D(x)|\phi|\,dx\\
\leq|D|_{L^q}|\phi|_{L^{q'}}\leq\epsilon |\phi|_{L^{q'}}^2+\frac1{4\epsilon}|D|_{L^q}^2\leq \epsilon M_{q'}^2 |\phi|_{H^1}^2+\frac1{4\epsilon}|D|_{L^q}^2;
\end{multline*}
choosing $\epsilon:= \lambda_0/(2M_{q'}^2)$ we get property (1).
In order to prove (2), we notice that, since $\mathcal L$ is a
Lyapunov functional for $\pi$ and $\mathcal A$ is compact in
$H^1_0(\Omega)$, there exists an equilibrium $\phi$ such that, for
every $u\in\mathcal A$,
\begin{multline*}
\int_\Omega|\nabla u|^2\,dx+\int_\Omega\beta(x)|u|^2\,dx-\int_\Omega F(x,u)\,dx\\
\leq\int_\Omega|\nabla \phi|^2\,dx+\int_\Omega\beta(x)|\phi|^2\,dx-\int_\Omega F(x,\phi)\,dx.
\end{multline*}
Then, for $\epsilon>0$, we have:
\begin{multline*}
\lambda_0|u|_{H^1}^2\leq \int_\Omega D(x)|u|\,dx+\Lambda_0|\phi|_{H^1}^2+\int_\Omega F(x,\phi)\,dx\\
\leq  \epsilon M_{q'}^2 |u|_{H^1}^2+\frac1{4\epsilon}|D|_{L^q}^2+\Lambda_0|\phi|_{H^1}^2+\int_\Omega F(x,\phi)\,dx.
\end{multline*}
We choose $\epsilon:= \lambda_0/(2M_{q'}^2)$ and the conclusion follows.
\end{proof}

Finally, we have:

\begin{Theorem}\label{special1} Assume that Hypotheses \ref{hyp2} and
\ref{hyp5} are satisfied. Then  for every $0<\epsilon\leq 1$,
\begin{equation*}
\partial_u f(x,0)\leq \frac2\epsilon D(x)+\frac\epsilon2
C(1+\epsilon^\gamma).
\end{equation*}
\end{Theorem}
\begin{proof}
For $\epsilon>0$ we have:
\begin{equation*}
f(x,\epsilon)=f(x,0)+\partial_u f(x,0)\epsilon+\int_0^\epsilon\left(\int_0^s\partial_{uu}f(x,r)\,dr\right)\,ds.
\end{equation*}
It follows that
\begin{equation*}
f(x,0)\epsilon+\partial_u f(x,0)\epsilon^2+\epsilon\int_0^\epsilon\left(\int_0^s\partial_{uu}f(x,r)\,dr\right)\,ds=f(x,\epsilon)\epsilon\leq D(x)\epsilon.
\end{equation*}
Therefore
\begin{equation*}
\partial_u f(x,0)\leq \frac{D(x)+|f(x,0)|}\epsilon+\frac1\epsilon\int_0^\epsilon\left(\int_0^sC(1+|r|^\gamma)\,dr\right)\,ds,
\end{equation*}
and the conclusion follows.
\end{proof}

\begin{Rem}Inequality (3) in Theorem \ref{special1} shows that
Hypotheses \ref{hyp2} and \ref{hyp5} together imply Hypothesis
\ref{hyp3}, with $V_\epsilon(x)=\frac{2C}\epsilon D(x)$.
\end{Rem}

%--------------------------------------------------------------

\end{document}